\documentclass[10pt]{article} 
\usepackage{setspace}
\usepackage{float}
\usepackage[utf8]{inputenc} 

\usepackage[margin=1in]{geometry}
\usepackage{graphicx} 
\usepackage{epstopdf}

\usepackage{booktabs} 
\usepackage{array} 
\usepackage{paralist} 
\usepackage{verbatim} 

\usepackage{fancyhdr} 
\usepackage{sectsty}
\allsectionsfont{\sffamily\mdseries\upshape} 

\usepackage[nottoc,notlof,notlot]{tocbibind} 
\usepackage[titles]{tocloft} 

\usepackage{amsmath, amsthm, amssymb}

\usepackage{multirow}
\usepackage{algorithm}
\usepackage{algorithmic}

\usepackage{xcolor}
\usepackage{booktabs}
\usepackage{comment}
\usepackage{amsmath}
\usepackage{mathrsfs}
\usepackage{float}
\usepackage{epstopdf}
\usepackage{multirow}
\usepackage{rotating}
\usepackage{amssymb}

\usepackage{subcaption}

\usepackage{amsthm}
\newtheorem{thm}{Theorem}
\newtheorem{lem}{Lemma}
\newtheorem{prop}{Proposition}
\newtheorem{cor}{Corollary}

\newtheorem{dfn}{Definition}

\usepackage{tikz,pgfplots}
\usetikzlibrary{matrix}
\usepgfplotslibrary{groupplots}
\pgfplotsset{compat=newest}
\usetikzlibrary{calc,positioning,shapes,fit,arrows,shadows}
\tikzstyle{line} = [ draw, -latex']  
\usetikzlibrary{shapes,arrows}
\usepgfplotslibrary{units}
\usepackage{url}
\usepackage{color}

\graphicspath{{figures/}} 

\usepackage{color}

\newcommand{\cI}{\mathcal{I}}

\newcommand{\cE}{\mathcal{E}}
\newcommand{\cH}{\mathcal{H}}
\newcommand{\cD}{\mathcal{D}}

\usepackage{mathabx}

\newcommand{\setM}{\mathbb{M}}

\allowdisplaybreaks

\title{Generator Subadditive Functions for  Mixed-Integer Programs} 
\author{Gustavo Angulo\thanks{gangulo@ing.puc.cl, Department of Industrial and Systems Engineering, Pontificia Universidad Católica de Chile, Santiago, Chile 7820436.} \and
Burak Kocuk\thanks{burakkocuk@sabanciuniv.edu, Industrial Engineering Program,  Sabanc{\i} University, Istanbul, Turkey 34956.} \and
Diego A. Mor{\'{a}}n R.\thanks{morand@rpi.edu, Industrial and Systems Engineering Department,  Rensselaer Polytechnic Institute, Troy, NY USA 12180-3590.}   
} 

\date{} 

\usepackage{color}

\def \LF{ {\mathcal{L}}}

\def \FF{ {\mathcal{F}}}

\def \rr{ {\mathbb{R}}}
\def \zz{ {\mathbb{Z}}}
\def \qq{ {\mathbb{Q}}}

\def \genX{ {\mathcal{M}}}
\def \genXsubsetSleq{S}
\def \genXcont{ {\mathcal{R}}}

\def \conv{\textup{conv}}

\def \int{ {\textup{int}}}

\def \tq{{\,:\,}}

\def \RR{\mathbb{R}}

\def \fancymu{{\hat\mu}}

\newcommand{\BK}[1]{{\color{red} [BK] #1}}

\begin{document}
\maketitle

\begin{abstract}
For equality-constrained linear mixed-integer programs (MIP) defined by rational data, it is known that the subadditive dual is a strong dual and that there exists an optimal solution of a particular form, termed {\em generator subadditive function}. Motivated by these results, we explore the connection between Lagrangian duality, subadditive duality and generator subadditive functions for general equality-constrained MIPs where the vector of variables is constrained to be in a monoid. We show that  strong duality holds via generator subadditive functions under certain conditions. For the case when the monoid is defined by the set of all mixed-integer points contained in a convex cone, we show that strong duality holds under milder conditions and over a more restrictive set of dual functions. Finally, we provide some examples of applications of our results.
\end{abstract}





\section{Introduction}
\label{sec:intro}

Consider the following equality-constrained mixed-integer program (MIP)
\begin{equation}\label{eq:genX_MIP}
\begin{aligned}
z^* :=   \inf  &\hspace{0.5em}  c^Tx + d^T y \\
  \mathrm{s.t.}   &\hspace{0.5em} A x + Gy = b  \\
  & \hspace{0.5em}  (x,y)  \in  \genX,
\end{aligned}
\end{equation}
where {$c \in \RR^{n_1}$, $d \in \RR^{n_2}$}, $A \in \RR^{m \times n_1}$, $G \in \RR^{m \times n_2}$, $b \in \RR^{m}$ and $\genX\subseteq\zz^{n_1}\times\rr^{n_2}$ with the usual $+$ operation is a {\em monoid}, that is, $(0,0)\in\genX$, $(x_1,y_1),(x_2,y_2)\in \genX$ implies $(x_1,y_1)+(x_2,y_2)\in\genX$ and $+$ is associative in~$\genX$. 
%

In this paper, we study certain dual problems associated to the primal MIP~\eqref{eq:genX_MIP} for several special cases.  Duality is a well-studied and fundamental concept in optimization theory and methodology \cite{johnson1974group, jeroslow1978cutting,Tind1981,Guzelsoy, Moran, kocuk2019subadditive}. For a minimization type \textit{primal} problem, a feasible  solution to its \textit{dual} problem can provide a lower bound to the optimal primal value, hence, it enables to certify the quality of a feasible solution. In addition, for a {\em strong dual} problem, its optimal value matches the optimal value of the primal problem, which can be used as a certificate of optimality. 

Encoding the dual problem can be quite different for different classes of  primal optimization problems. For example, the dual problem of a linear program (LP) or a conic program (CP) is another linear program or conic program,  involving the dual cone in the latter case (see, e.g., \cite{BenTal01}). Hence, these dual problems are finite dimensional optimization problems with similar representability as their primal counterparts. However, when integer variables are involved, the situation changes drastically. In the case of linear or conic MIPs, one typically resorts to a \textit{subadditive} dual, whose feasible region is defined over
subadditive functions, 
rather than finite dimensional elements as in the case of LPs and CPs (see,  \cite{johnson1974group, jeroslow1978cutting, Guzelsoy} for linear MIPs, and \cite{Moran, kocuk2019subadditive, AJAYI2020329} for conic inequality-constrained MIPs). Hence, it is incomparably more difficult to encode the dual problem when integer variables are involved.
As an example, the subadditive dual problem for the equality-constrained linear MIP, where we have $\genX= \mathbb{Z}_+^{n_1} \times \mathbb{R}_+^{n_2}$  in~\eqref{eq:genX_MIP}, is given as follows:\begin{equation}\label{eq:subadditiveDualLinear}
    \begin{aligned}
  \rho^* :=  \sup  &\hspace{0.5em}  f(b) \\
      \mathrm{s.t.}   &\hspace{0.5em} f(a^j) \le c_j &j=1,\dots,n_1 \\
      &\hspace{0.5em} \bar f(g^j) \le d_j &j=1,\dots,n_2 \\
      & \hspace{0.5em}  f(0) = 0, \ f:\rr^m\to\rr\ \text{ is subadditive}.
    \end{aligned}
    \end{equation}
Here, $a^j$ and $g^j$ respectively denote the $j$-th column of matrices $A$ and $G$.
The following result states that for feasible {linear} MIPs with rational data, the subadditive dual is a strong dual,  that is, the duality gap is zero (i.e., $\rho^*=z^*$) and the subadditive dual is solvable (i.e., there exists a feasible dual function $f$ such that $\rho^*=f(b)$). 
\begin{prop}[\cite{johnson1974group, jeroslow1978cutting}]\label{prop:strongSubadditiveLinear}
    Consider a feasible linear MIP (i.e., $\genX= \mathbb{Z}_+^{n_1} \times \mathbb{R}_+^{n_2}$) of the form \eqref{eq:genX_MIP}, where
      $A\in\qq^{m \times n_1}$, {$G\in\qq^{m \times n_2}$} and $b \in \qq^m$,
    and the subadditive dual as defined in~\eqref{eq:subadditiveDualLinear}.
    Then, the subadditive dual~\eqref{eq:subadditiveDualLinear} is a strong dual to linear MIP~\eqref{eq:genX_MIP}.
\end{prop}

There have been some limited attempts to solve the subadditive dual problem in the case of linear MIPs for special problem classes  as subadditive duals do not admit practical solution procedures. For example, \cite{klabjan2004practical, Klabjan2007} consider a class of problems with equality constraints, rational and nonnegative data, and nonnegative integer variables. In this case, it is shown that the optimal function can be written as the value function of a parametric packing type of linear IP, whose feasible region is a relaxation of the original feasible region. Then, an iterative algorithm is proposed to obtain the optimal parameters. Later, \cite{cheung2016certificates} generalizes the setting to the case where there is no nonnegativity assumption on the data and continuous variables are present. 
In all these three papers, the optimal dual solution has a specific functional form as described below.
\begin{prop}[\cite{Klabjan2007, cheung2016certificates}]\label{prop:generatorLinear}
    Consider a feasible linear MIP 
    under the same assumptions of Proposition~\ref{prop:strongSubadditiveLinear}.
    Let $\Omega := \{A x + Gy\tq (x,y) \in \mathbb{Z}_+^{n_1} \times \mathbb{R}_+^{n_2} \}$ be the set of r.h.s. that make the primal problem~\eqref{eq:genX_MIP} feasible. 
    Then, there exists a vector $\alpha\in\rr^m$ for which the optimal function $F_\alpha : \Omega \to \mathbb{R}$ is of the form
    \begin{equation*} 
    F_\alpha(\omega) = \alpha^T \omega - \max \left \{ (A^T\alpha - c)^T x + (G^T\alpha - d)^T y : Ax + G y \le \omega, \ (x,y) \in \mathbb{Z}_+^{n_1} \times \mathbb{R}_+^{n_2} 
    \right \}.
    \end{equation*}
\end{prop}
The functions appearing in Proposition~\ref{prop:generatorLinear} are called {\em generator subadditive functions} and, in principle, are still challenging to compute as the formula requires to solve a linear MIP. Due to the following result, the size of this linear MIP can be decreased and the computational burden can be  reduced.
\begin{prop}[\cite{Klabjan2007, cheung2016certificates}]\label{prop:generatorLinear-withE}
    Consider a feasible linear MIP 
    under the same assumptions of Proposition~\ref{prop:generatorLinear}. Let $\alpha\in\rr^m$ and define the sets $E_x :=\{ j : \alpha^T a^j > c_j \text{ or } a^j \not\ge0\}$ and $E_y :=\{ j :  \alpha^T g^j > d_j \text{ or } g^j \not\ge0 \}$. Then,
    \begin{equation*} 
    F_\alpha(\omega) = \alpha^T \omega - \max \left \{ \sum_{j \in E_x } (\alpha^T a^j - c_j) x_j + \sum_{j\in E_y} ( \alpha^T g^j - d_j) y_j : \sum_{j \in E_x } a^j  x_j + \sum_{j\in E_y} g^j y_j \le \omega, \ (x,y)  \in \mathbb{Z}_+^{n_1} \times \mathbb{R}_+^{n_2}  \right \}.
    \end{equation*}
\end{prop}
We would like to note that the above approaches are not the only attempts to solve the subadditive dual for linear MIPs. 
{A piecewise linear approximation of the subadditive dual problem of the mixed-integer program of binarized neural networks is utilized in~\cite{aspman2024taming}.}

Our paper extends the main results in the three papers~\cite{klabjan2004practical, Klabjan2007, cheung2016certificates} mentioned above to the case of general equality-constrained  MIPs~\eqref{eq:genX_MIP} as follows: 
%
\begin{itemize}
    \item 
 {Strong subadditive duality for conic MIPs:} We generalize Proposition~\ref{prop:strongSubadditiveLinear} in Proposition~\ref{prop:strongDual}. 
    \item 
 {Optimal generator subadditive functions:} We generalize Proposition~\ref{prop:generatorLinear} in Theorem~\ref{thm:strong+generator}.
\item 
 {Reducing the size for  conic MIPs with block structure:} We generalize Proposition~\ref{prop:generatorLinear-withE} in Proposition~\ref{prop:generatorLinear-withE-block}.
\end{itemize}




{
The rest of the paper is structured as follows: In Section~\ref{sec:duality-MIP}, we explore the connections between Lagrangian duality, subadditive duality and generator subadditive functions for  equality-constrained  MIPs where the vector of variables comes from a monoid. In Section~\ref{sec:duality-conic-MIP}, we specialize our results when the monoid is defined as the set of the mixed-integer points contained in a regular cone.
Finally, we provide  illustrations of our results in Section~\ref{sec:illustration}.
}

\section{Duality for equality-constrained mixed-integer programs}
\label{sec:duality-MIP}

The paper~\cite{Klabjan2007} established the relationship between {\em generator subadditive functions} with Lagrangian and subadditive duality for integer linear programs. In this section, we generalize these results for problem~\eqref{eq:genX_MIP} and the associated dual problems.

\subsection{Lagrangian duality and generator subadditive functions}
\label{sec:main-Lagr}

Let $\Omega^\leq=\{\omega\tq A x + Gy\leq \omega\ \textup{for some}\ (x,y) \in \genX\}$ and for $\omega\in \Omega^\leq$ define
$\genX^\leq(\omega)=\{(x,y) \in \genX \tq Ax+Gy\leq \omega\}$. Suppose $\genXsubsetSleq\subseteq \genX^\leq(b)$ and consider the optimization problem
\begin{equation}\label{eq:genX_MIP_lagrange}
\begin{aligned}
z^\leq :=   \inf  &\hspace{0.5em}  c^Tx + d^T y \\
  \mathrm{s.t.}   &\hspace{0.5em} A x + Gy = b  \\
  & \hspace{0.5em}  (x,y)  \in  \genXsubsetSleq.
\end{aligned}
\end{equation}
For  $\alpha\in\rr^m$, define $\LF(\alpha):=\inf\{c^Tx+d^Ty+\alpha^T(b-Ax-Gy)\tq(x,y)\in \genXsubsetSleq\}$.
Since we are dualizing equality constraints, it is easy to see that the value $\LF(\alpha)$ gives a lower bound for the optimal value of \eqref{eq:genX_MIP_lagrange}, that is,
\begin{equation}\label{eq:LagrangianWeakDuality}
\LF(\alpha)\leq \inf\{c^Tx+d^Ty\tq Ax+Gy= b,\ (x,y)\in \genXsubsetSleq\}.
\end{equation}
For a set $X\subseteq\rr^n$, we denote 
its convex hull as $\conv(X)$. 
In the Lagrangian duality theory (e.g., see~\cite{DMM2024,geoffrion2010lagrangian} for the integer programming case), there are sufficient conditions for the existence of $\alpha^*\in\rr$ such~that
\begin{equation}\label{GeoffrionsTheoremHolds}
    \LF(\alpha^*)=\inf\{c^Tx+d^Ty\tq Ax+Gy= b,\ (x,y)\in \conv(\genXsubsetSleq)\}.
\end{equation}
Note that $\alpha^*$ satisfies  $\LF(\alpha^*)=\sup\{\LF(\alpha)\tq \alpha\in\rr^m\}$, and thus is an optimal solution to the Lagrangian dual problem. In principle, the bound in~\eqref{GeoffrionsTheoremHolds} may be strictly less than the optimal value of problem~\eqref{eq:genX_MIP_lagrange} since the minimization is over $\conv(\genXsubsetSleq)\supseteq\genXsubsetSleq$. Nevertheless, we will show in Theorem~\ref{thm:opt_convX_equals_X} below that 
\begin{equation}\label{eq:opt_convX_equals_X}
\inf\{c^Tx+d^Ty\tq Ax+Gy= b,\ (x,y)\in \conv(\genXsubsetSleq)\}=\inf\{c^Tx+d^Ty\tq Ax+Gy= b,\ (x,y)\in \genXsubsetSleq\}.
\end{equation}
Therefore, if \eqref{GeoffrionsTheoremHolds} is satisfied, then any optimal solution $\alpha^*$ to the Lagrangian dual problem satisfies:
$$\LF(\alpha^*)=\inf\{c^Tx+d^Ty\tq Ax+Gy= b,\ (x,y)\in \genXsubsetSleq\}.$$
In order to prove Theorem~\ref{thm:opt_convX_equals_X}, we need the following lemma.  


\begin{lem} \label{lemFintegerhull}
  Let $\setM\subseteq \rr^n$ be a set and let  $X\subseteq \{x \in \setM\tq Ax\leq b\}$ be a  set such that $\conv(X)\cap \setM=X$. Assume that $T=\{x\in \conv(X)\tq Ax= b\}$ is a nonempty face of $\conv(X)$. Then $\conv(X)\cap T=\conv(T\cap \setM)$.  
\end{lem}
\begin{proof}
    By assumption and since $X\subseteq \setM$, we have $\conv(X)\cap \setM=X=X\cap \setM$. Therefore,
    $$\conv(X)\cap T=\conv(X\cap \setM)\cap T=\conv(\conv(X)\cap \setM)\cap T=\conv(T\cap \setM),$$
    \noindent where the last equality follows from the well-known fact that for any set $\setM'$, convex set $X'\subseteq \rr^n$ and face $T'$ of $X'$, we have $\conv(X'\cap \setM')\cap T'=\conv(T'\cap \setM')$.
\end{proof}

\begin{thm}\label{thm:opt_convX_equals_X}
    Assume that $\conv(\genXsubsetSleq)\cap \genX
    = \genXsubsetSleq$
    and  that $T=\{(x,y) \in\rr^{n_1}\times\rr^{n_2}\tq Ax+Gy= b,\ (x,y) \in \conv(\genXsubsetSleq)\}$ is nonempty. Then, equation~\eqref{eq:opt_convX_equals_X} is satisfied.
\end{thm}

\begin{proof}
Denoting $z^{\conv}=\inf\{c^Tx+d^Ty\tq Ax+Gy= b,\ (x,y)\in \conv(\genXsubsetSleq)\}$, we have  \begin{align*}
    z^{\conv}&=\inf\{c^Tx+d^Ty\tq(x,y)\in \conv(\genXsubsetSleq)\cap T\}
    =\inf\{c^Tx+d^Ty\tq(x,y)\in \conv(T\cap\genX)\}\\
    &=\inf\{c^Tx+d^Ty\tq(x,y) \in T\cap\genX\}
    =\inf\{c^Tx+d^Ty\tq Ax+Gy= b,\ (x,y) \in \genXsubsetSleq\}=z^\leq,
\end{align*}
\noindent where the first equality follows by definition of $T$, the second equality follows from applying Lemma \ref{lemFintegerhull} to $X=\genXsubsetSleq$ and $\setM=\genX$, the third equality follows from the fact optimizing a linear function over a set is equivalent to optimize over its convex hull, and the last equality follows from the definition of $T$ and $\conv(\genXsubsetSleq)\cap \genX = \genXsubsetSleq$.
\end{proof}

Note that $\conv(\genX^\leq(\omega))\cap \genX= \genX^\leq(\omega)$ for any $\omega\in\Omega$ since $\genX^\leq(\omega)$ are points in $\genX$ contained in  the convex set $\{(x,y) \in \rr^{n_1}\times\rr^{n_2} \tq Ax+Gy\leq \omega\}$. Assuming that the feasible region of \eqref{eq:genX_MIP} is nonempty and that \eqref{GeoffrionsTheoremHolds} holds, we can apply Theorem~\ref{thm:opt_convX_equals_X}  to $\genXsubsetSleq=\genX^\leq(b)$ and obtain that there exist $\alpha^*\in\rr$ such that $\LF(\alpha^*)=\inf\{c^Tx+d^Ty\tq Ax+Gy= b,\ (x,y) \in \genX^\leq(b)\}$.
The relationship of Lagrangian duality and {\em generator subadditive functions} is established by noting that
\begin{align*}
    \LF(\alpha^*)&=\inf\{c^Tx+d^Ty+{\alpha^*}^T(b-Ax-Gy)\tq(x,y)\in \genX^\leq(b)\}\\
    &={\alpha^*}^Tb+\inf\{(c-A^T\alpha^*)^Tx+(d-G^T\alpha^*)^Ty\tq(x,y)\in \genX^\leq(b)\}\\
    &={\alpha^*}^Tb-\sup\{(A^T\alpha^*-c)^T x+(G^T\alpha^*-d)^Ty\tq(x,y)\in \genX^\leq(b)\} .
\end{align*}

We generalize the definition in~\cite{Klabjan2007} as follows.

\begin{dfn}[Generator subadditive function]
    Consider a feasible  MIP of the form~\eqref{eq:genX_MIP}. 
    Then,  for $\alpha \in \rr^m$, we define a generator subadditive function
    $F_\alpha:\Omega^\leq \to \rr\cup\{-\infty\}$ as
     \begin{equation} \label{eq:generator} F_\alpha(\omega) = \alpha^T \omega - \sup \left \{ (A^T\alpha - c)^T x + (G^T\alpha - d)^T y\, :\, (x,y)\in \genX^\leq(\omega)\right \}.
        \end{equation}
\end{dfn}
Note that by the previous discussion on Lagrangian duality from~\eqref{eq:LagrangianWeakDuality}, we have that
\begin{equation}\label{eq:GenSadWeakDuality}
F_\alpha(b)\leq \inf\{c^Tx+d^Ty\tq Ax+Gy= b,\ (x,y)\in \genX^\leq(b)\}=z^*,
\end{equation}
\noindent and therefore we can use the function $F_\alpha$ to find a lower bound for the optimal value of the primal problem~\eqref{eq:genX_MIP} -- this property is related to the subadditive dual being a {\em weak dual}, we will see this in the next section.

\subsection{Subadditive duality}
In order to state the subadditive dual for the primal problem~\eqref{eq:genX_MIP}, we need some definitions.

\begin{dfn}[Subadditive  function]
A function $f:\cD \to \RR$ is subadditive if $f(u+v) \le f(u) + f(v)$ for all $u,v \in \cD$ such that $u+v\in \cD$. 
We denote the set of subadditive functions $f:\cD\to\rr$ 
as $\mathcal{F}_\cD$ and for $f\in\mathcal{F}_\cD$, we define $\bar f(x) := \lim\sup_{\delta\to0^+}\frac{f(\delta x )}{\delta}$.
\end{dfn}

\begin{dfn}[Integral generating set]
Let $\genX$ with the usual $+$ operation be a monoid. A generating set of $\genX$ is a set $\cI(\genX)\subseteq \genX$ such that for any $w \in\genX$ there exists $h_1,\ldots,h_p\in \cI(\genX) $ and $\lambda_1,\ldots,\lambda_p\in\zz_+$ such that
 $ w=\sum_{i=1}^p\lambda_ih_i$.
\end{dfn}

For a general monoid, integral generating sets are not finite. Conditions for finiteness of the generating set are given in~\cite{jeroslowmonoids1978,hemmecke2007representation,DM2013}. 

Let $\Omega=\{A x + Gy\tq (x,y) \in \genX\}$ be 
the set of r.h.s. that make the primal problem~\eqref{eq:genX_MIP} feasible. 
Let $\cD\subseteq \rr^m$ such that $\cD\supseteq\Omega$. Consider the following subadditive dual problem of~\eqref{eq:genX_MIP}:
\begin{equation}\label{eq:dual_genX_MIP}
    \begin{aligned}
    \rho^* := \sup  &\hspace{0.5em}  f(b) \\
      \mathrm{s.t.}   &\hspace{0.5em} f(A u+Gv) \le c^T u + d^Tv& (u,v) \in \cI(\genX )\\
      & \hspace{0.5em}  f(0) = 0, \ f \in \mathcal{F}_\cD.
    \end{aligned}
    \end{equation}

The following lemma shows that the choice of the generating set does not change the associated constraints in the dual problem.
\begin{lem}\label{lem:IntGenSet-indifferent} Let $\cI(\genX )$, $\cI'(\genX )$ be two generating sets of $\genX$. Then, for any $f\in\FF_D$ with $f(0)=0$, we have that
 $f(A u+Gv) \le c^T u + d^Tv\ \textup{for all}\ (u,v) \in \cI(\genX )$
 is equivalent to $f(A u+Gv) \le c^T u + d^Tv\ \textup{for all}\  (u,v) \in \cI'(\genX )$.
\end{lem}
\begin{proof}
   Assume that $f(A u+Gv) \le c^T u + d^Tv\ \textup{for all}\  (u,v) \in \cI(\genX )$ and let $(u',v') \in \cI'(\genX )$. Then there exists  $(u_1,v_1),\ldots,(u_p,v_p)\in \cI(\genX) $ and $\lambda_1,\ldots,\lambda_p\in\zz_+$ such that $(u',v')=\sum_{i=1}^p\lambda_i(u_i,v_i)$. We have
   $$f(Au'+Gv')=  f\left( A  \sum_{i=1}^p \lambda_i u_i  + G  \sum_{i=1}^p \lambda_i v_i  \right)  \le  \sum_{i=1}^p \lambda_i f \left( A   u_i + G   v_i  \right) \le \sum_{i=1}^p \lambda_i (c^T u_i + d^T v_i)=  c^T u'+d^Tv',$$
\noindent where the first inequality follow from subadditivity of $f$ and $f(0)=0$. Therefore, if the function $f$ satisfies the constraints for the integral generating set  $\cI(\genX )$ then it also satisfies them for the integral generating set  $\cI'(\genX )$ (and viceversa).
\end{proof}

We show next that weak duality holds for the primal~\eqref{eq:genX_MIP} and dual~\eqref{eq:dual_genX_MIP}.

\begin{prop}[Weak duality]\label{prop:weakDual}
 For any vector $(x,y)$ feasible for~\eqref{eq:genX_MIP} and for any function $f$ feasible for~\eqref{eq:dual_genX_MIP},  we have that $f(b)\leq c^Tx+d^Ty$.
\end{prop}
\begin{proof}
For $(x,y)$ feasible for~\eqref{eq:genX_MIP}, we have that $b=Ax+Gy$. Thus, the fact that $f(b)\leq c^Tx+d^Ty$ follows from $f(b)=f(Ax+Gy)$ and Lemma~\ref{lem:IntGenSet-indifferent} applied to $\cI'(\genX)=\genX$.
\end{proof}

Sufficient conditions for functional  duals like the dual~\eqref{eq:dual_genX_MIP} to be a {\em strong dual}  have been given in~\cite{Tind1981, Moran,kocuk2019subadditive,Flippo1996,johnson1974group,jeroslow1979minimal, santana2017some}. A simple condition 
to have strong duality is to consider functions defined on $\cD=\Omega$  as   the so-called {\em value function} $\nu_\genX$ defined for $ \omega\in\Omega$ as   
$$\nu_\genX(\omega)=\inf\{c^Tx+d^Ty\tq Ax+Gy=\omega, (x,y) \in \genX\},$$
\noindent is an optimal solution to the dual problem whenever $\nu_\genX(\omega)>-\infty$ for all $\omega\in\Omega$ (for instance, see~\cite{Tind1981}). 
On the other hand, when $\cD=\rr^m$, several sufficient conditions to have strong duality have been studied in~\cite{Moran,kocuk2019subadditive}.

We now show  that under some conditions, {\em generator subadditive functions} are feasible and/or optimal for the subadditive dual of \eqref{eq:genX_MIP} when $\cD=\Omega^\leq$. We note here that dual feasibility requires that the functions belong to the set $\FF_{\Omega^\leq}$ and hence we must have that the function $F_\alpha:\Omega^\leq \to \rr$ is well-defined, that is, $F_\alpha(\omega)>-\infty$ for all $\omega\in \Omega^\leq$. A sufficient condition for this to happen is   the boundedness of sets $\genX^\leq(b)$  for any $b\in \Omega^\leq$. However, more general conditions can be given. For instance, when problem~\eqref{eq:genX_MIP} is a conic MIP, a sufficient condition is the feasibility of its subadditive dual (see~\cite{kocuk2019subadditive}). 

The next proposition recovers results from \cite{Klabjan2007,cheung2016certificates}.

\begin{prop}\label{prop:generator}
    Consider a feasible conic MIP of the form~\eqref{eq:genX_MIP} and the subadditive dual of the form~\eqref{eq:dual_genX_MIP} with $\cD=\Omega^\leq$. 
    Then, if $F_\alpha:\Omega^\leq \to \rr$ is well-defined, then $F_\alpha$ is a feasible dual solution.         
\end{prop}

\begin{proof}
           Firstly,  we check  that {\em generator subadditive functions} are indeed subadditive functions. Let $\omega_1,\omega_2\in\Omega^\leq$. Observe that if $(x_1,y_1)\in \genX^\leq(\omega_1)$ and $(x_2,y_2)\in \genX^\leq(\omega_2)$, then $(x_1+x_2,y_1+y_2)\in \genX^\leq(\omega_1+\omega_2)$, and therefore $\omega_1+\omega_2\in\Omega^\leq$. Moreover, we obtain
\begin{align*}
    \sup \left \{ (A^T\alpha - c)^T x + (G^T\alpha - d)^T y\, :\, (x,y)\in \genX^\leq(\omega_1+\omega_2)\right \}&\geq \sup \left \{ (A^T\alpha - c)^T x + (G^T\alpha - d)^T y\, :\, (x,y)\in \genX^\leq(\omega_1)\right \}\\
    &+\sup \left \{ (A^T\alpha - c)^T x + (G^T\alpha - d)^T y\, :\, (x,y)\in \genX^\leq(\omega_2)\right \},
\end{align*}
\noindent and hence we conclude  $F_\alpha(\omega_1+\omega_2)\leq F_\alpha(\omega_1)+F_\alpha(\omega_2)$.
           
Notice that for any $(u,v) \in \genX$, we have 
\begin{equation}\label{dual_constraint_S}
        \begin{split}
         F_\alpha(A u+G v) & =   \alpha^T (A u+G v) - \sup \left \{ (A^T\alpha - c)^T x + (G^T\alpha - d)^T y \tq (x,y)\in \genX^\leq(A u+G v) \right \} \\
         & \le  
     \alpha^T Au -   (A^T\alpha - c)^T u +\alpha^T Gv -   (G^T\alpha - d)^T v  = c^T u + d^T v,
         \end{split}
    \end{equation}
   \noindent where the inequality follows from the fact that $(u,v)\in \genX^\leq(A u+G v)$. Since $\cI(\genX)\subseteq\genX$, this shows that $F_{\alpha}$ satisfies the first constraint in the dual~\eqref{eq:dual_genX_MIP}.
   Finally, since $\genX$ is a monoid, we have that $(u,v)=(0,0)\in\genX$ and thus from~\eqref{dual_constraint_S} it follows that $F_\alpha(0) \le 0$. Since any subadditive function satisfies the reverse inequality, we conclude that $F_\alpha(0)=0$.
\end{proof}

We now review some definitions related to cones.
\begin{dfn}[Regular cone, Conic inequality, Dual cone]
	We call a cone $K \subseteq \RR^m$  regular if it is closed, convex, pointed and full-dimensional.
The conic inequality with respect to $K$ is defined as $x \succeq_{K} y$ if and only if  $x - y \in K$.
The dual cone to a cone $K \subseteq \RR^m$ is defined as $K_* := \{y\in \mathbb{R}^m: x^Ty \ge 0, \ \textup{for all}\ x \in K\}$.
 
\end{dfn}

\begin{thm}[Strong duality via generator subadditive functions]\label{thm:strong+generator}
    Consider a feasible   MIP of the form~\eqref{eq:genX_MIP} and the subadditive dual of the form~\eqref{eq:dual_genX_MIP} with $\cD=\Omega^\leq$. Assume $F_\alpha:\Omega^\leq \to \rr$ is well-defined, and that $\conv(\genX^\leq(b))$ is conic representable, that is,
    $$\conv(\genX^\leq(b))=\{(x,y) \in\rr^{n_1}\times\rr^{n_2}\tq\exists\, w\in\rr^{n_3}\ \textup{s.t.}\ \Pi x  + \Phi y + \Psi w \succeq_{C} \pi\},$$
    for some matrices $\Pi, \Phi,\Psi$, a vector $\pi$ and a regular cone $C$. Assume further that the  (conic) dual of $\inf \{ c^T x + d^T y : \ Ax + Gy = b,\ \Pi x  + \Phi y + \Psi w \succeq_{C} \pi \}$ is a strong dual.
    Then, there exists $\alpha^*\in\rr^m$ such that the function $F_{\alpha^*}:\Omega^\leq \to \rr$ is well-defined and it is an optimal solution to the subadditive dual \eqref{eq:dual_genX_MIP}.
\end{thm}
\begin{proof} Let $\omega\in\Omega$. By Theorem~\ref{thm:opt_convX_equals_X}, we obtain that
    \begin{align*}
    z^*(\omega):=&\inf\{c^Tx+d^Ty\tq Ax+Gy= b,\ (x,y) \in \genX^\leq(\omega)\} \\
    =&\inf\{c^Tx+d^Ty\tq Ax+Gy= b,\ (x,y)\in \conv(\genX^\leq(\omega))\},
\end{align*}
\noindent and by assumption we conclude that
    \begin{equation}\label{eq:IP with valid conic repr}
    \begin{aligned}
    z^*(b)=   \inf  \{  c^Tx + d^T y  : 
        A x + G y = b , \ \Pi x  + \Phi y + \Psi w \succeq_{C} \pi \}.
    \end{aligned}
    \end{equation}
    
    Consider the conic dual problem associated to \eqref{eq:IP with valid conic repr}:
    \begin{equation*} 
    \begin{aligned}
    \sup \{ b^T \alpha + \pi^T \gamma  : 
       A^T \alpha +{\Pi}^T \gamma = c  , \ 
        G^T \alpha + {\Phi}^T \gamma = d , \ 
        {\Psi}^T \gamma = 0 ,  \
        \gamma \in C_* \}.
    \end{aligned}
    \end{equation*}    
Since strong duality holds,   for any optimal pair of dual variables $(\alpha^*, \gamma^*)$, we have $z^*= b^T \alpha^* + \pi^T \gamma^*$. Therefore, for $(x,y)\in \genX^\leq(b)$ and $w$ such that ${\Pi} x  +  {\Phi} y + \Psi w\succeq_{C} \pi$, we obtain
    \begin{equation*}
    \begin{split}
    (A^T\alpha^* - c)^T x  + (G^T\alpha^* - d)^T y 
     = & \left(   -{\Pi}^T \gamma^*  \right)^T x  + 
     \left(-{\Phi}^T \gamma^*  \right)^T y
     +   \left(  -{\Psi}^T \gamma^*  \right )^T w \\
     = &   -\left( {\Pi} x  +  {\Phi} y + \Psi w \right)^T \gamma^* \\
    \le & -\pi^T  \gamma^* =  b^T \alpha^* - z^*,
    \end{split}
    \end{equation*}
    which implies
    $$
    \sup \left \{ (A^T\alpha^* - c)^T x + (G^T\alpha^* - d)^T y \tq  \ (x,y) \in \genX^\leq(b)  \right \} \le  b^T \alpha^* - z^*.
    $$
  As by definition  $F_{\alpha^*}(\omega)= {\alpha^*}^T \omega - \sup \left \{ (A^T\alpha^* - c)^T x + (G^T\alpha^* - d)^T y\, :\, (x,y)\in \genX^\leq(\omega)\right \}$, the inequality above shows that $F_{\alpha^*}(b)\ge z^*$. Hence, by \eqref{eq:GenSadWeakDuality}, we conclude that $F_{\alpha^*}(b)=z^*$. 
  Finally, since by assumption $F_{\alpha^*}:\Omega^\leq \to \rr$ is well-defined, by Proposition~\ref{prop:generator}, we obtain that $F_{\alpha^*}$ is feasible for the dual~\eqref{eq:dual_genX_MIP} with $\cD=\Omega^\leq$, and thus this generator subadditive function is an optimal solution to this dual problem. 
\end{proof}

\section{Duality for equality-constrained conic mixed-integer programs}
\label{sec:duality-conic-MIP}

\subsection{Strong duality}
We will show that~\eqref{eq:dual_genX_MIP} is a strong dual  when $\genX=K\cap(\mathbb{Z}^{n_1} \times \mathbb{R}^{n_2})$, where $K$ is a regular cone, and $\cD=\rr^m$.
We first need a definition.

\begin{dfn}[Nondecreasing function]
A function $f:\cD \to \RR$ is nondecreasing with respect to a regular cone $K \subseteq \RR^m$  if for $u,v\in\cD$ we have $u \succeq_K v \Rightarrow f(u) \ge f(v)$.
\end{dfn}
\begin{prop}[Strong duality]\label{prop:strongDual}
    Consider a feasible conic MIP of the form~\eqref{eq:genX_MIP} and the subadditive dual of the form~\eqref{eq:dual_genX_MIP}. Assume that the conic dual of the continuous relaxation of~\eqref{eq:genX_MIP} is feasible.  Then, there exists a dual feasible function $f$ such that $f(b)=\rho^* = z^*$.
\end{prop}
\begin{proof}
Let us consider the equivalent conic MIP
    \begin{equation}\label{eq:generic-equivalent}
    \begin{aligned}
    \hat z(\omega) := \inf  &\hspace{0.5em}  c^Tx + d^Ty\\
    \mathrm{s.t.}   
    &\hspace{0.5em} A x + Gy \ge \omega \\
    &               -A x - Gy \ge -\omega \\
    & \hspace{0.5em}  (x,y) \ge_K 0 \\
    & \hspace{0.5em}  x \in  \mathbb{Z}^{n_1}, y\in \mathbb{R}^{n_2},
    \end{aligned}
    \end{equation}
and its subadditive dual (see,~\cite{Moran, kocuk2019subadditive})
    \begin{equation}\label{eq:generic-equivalent-sad}
    \begin{aligned}
    \hat\rho(\omega) := \sup  &\hspace{0.5em}  H(\omega,-\omega,0) \\
      \mathrm{s.t.}   &\hspace{0.5em} H(a^j, -a^j, e^j) = -H(-a^j, a^j, -e^j)  = c_j\quad \text{for all}\ j=1,\dots,{n_1} \\
      &\hspace{0.5em} \bar H(g^j, -g^j, e^{n_1+j}) = -\bar H(-g^j, g^j, -e^{n_1+j})  = d_j \quad \text{for all}\ j=1,\dots,{n_2} \\
      & \hspace{0.5em}  H(0,0,0) = 0\\ 
      & \hspace{0.5em}  H:\rr^m\times\rr^m\times\rr^n\rightarrow \rr \text{ is subadditive and} \text{          nondecreasing w.r.t. } \rr_+^m\times\rr_+^m\times K.
    \end{aligned}
    \end{equation}
    Here, $e^i$ denotes the standard $i$th unit vector in $\rr^{n_1+n_2}$. 
For $\omega=b$, since the conic dual of the continuous relaxation of~\eqref{eq:genX_MIP} is feasible, and conic MIPs~\eqref{eq:genX_MIP} and \eqref{eq:generic-equivalent} are equivalent, the conic dual of the continuous relaxation of~\eqref{eq:generic-equivalent} is feasible. Then, due to~\cite{kocuk2019subadditive}, problem~\eqref{eq:generic-equivalent-sad} is a strong dual to problem~\eqref{eq:generic-equivalent}. Hence, there exists a feasible function ${\hat H}$ such that $\hat H(b,-b,0)=\hat \rho(b)=\hat z(b)$.  

Define the function $f(\omega) := \hat H(\omega,-\omega,0)$ for $\omega\in \rr^m$. We claim that $f$ is an optimal feasible solution to the subadditive dual~\eqref{eq:dual_genX_MIP}. The fact that $f$ is subadditive follows from the fact that $f$ is the composition of a subadditive function and a linear function.
Now, we show that  $f$ satisfies the  constraint in the dual~\eqref{eq:dual_genX_MIP}. Let $(u,v)  \in \cI(K\cap(\mathbb{Z}^{n_1} \times \mathbb{R}^{n_2}))$. Then, by considering $\omega=Au+Gv$ in  the primal~\eqref{eq:generic-equivalent} and the dual~\eqref{eq:generic-equivalent-sad}, by weak duality, we obtain $f(Au+Gv)\leq c^T u+d^Tv$. On the other hand, we have that $f(0) = \hat H(0,0,0) = 0$. So $f$ satisfies all constraints of the dual~\eqref{eq:dual_genX_MIP}. Finally, by construction $\hat z(b)=z^*$, and therefore, we obtain that $f(b) = \hat H(b,-b,0)=\hat z(b)=z^*$ and by Proposition~\ref{prop:weakDual}, we conclude that $f(b) =\rho^*=z^*$ as by weak duality, $z^*$ is an upper bound for $\rho^*$.
\end{proof}

The following is an immediate consequence of the strong duality result in Proposition~\ref{prop:strongDual}.
\begin{cor}[Complementary slackness]\label{cor:complSlack}
    Consider a feasible solution $(x,y)$ of conic MIP~\eqref{eq:genX_MIP} and a feasible dual solution $f$ of subadditive dual~\eqref{eq:dual_genX_MIP}.   Then, $ (x,y)$ and $ f$ are optimal solutions of the respective problems if and only if $f(Ax+Gy) = c^T x + d^T y$.
\end{cor}

\subsection{Conic MIPs where integer and continuous variables come from different cones}

In this section, we specialize our results 
for $\genX = (K_1 \cap \zz^{n_1}) \times K_2$, where  $K_1 \subseteq \rr^{n_1}$ and $K_2 \subseteq \rr^{n_2}$ are regular cones.
%
%
For this case, we propose the following subadditive dual problem of~\eqref{eq:genX_MIP}:

\begin{equation}\label{eq:dualGeneric-product}
    \begin{aligned}
    \rho^* := \sup  &\hspace{0.5em}  f(b) \\
      \mathrm{s.t.}   &\hspace{0.5em} f(A u) \le c^T u & u & \in \cI( K_1 \cap \zz^{n_1} )\\
      &\hspace{0.5em} \bar f(G v) \le d^T v & v & \in \cI ( K_2) \\
      & \hspace{0.5em}  f(0) = 0, \ f \in \mathcal{F}_\cD.
    \end{aligned}
    \end{equation}

Similar to the case of the dual problem~\eqref{eq:dual_genX_MIP}, the choice of generating set does not affect the feasible region of the dual problem~\eqref{eq:dualGeneric-product}. In other words, a result analogous to Lemma~\ref{lem:IntGenSet-indifferent} is true, but we do not state it in this paper.

When $\cD=\rr^m$, the subadditive dual~\eqref{eq:dualGeneric-product} generalizes results from the linear MIP literature \cite{johnson1974group, jeroslow1978cutting, Guzelsoy,Moran} since we can choose $\cI( K_1 \cap \zz^{n_1} )=\{e^1,\ldots,e^{n_1}\}$ and $\cI ( K_2)=\{e^{n_1+1},\ldots,e^{n_1+n_2}\}$, and thus we recover the constraints from the linear MIP dual \eqref{eq:subadditiveDualLinear}.

We will show in Proposition~\ref{prop:Dual equivalence} below that the dual~\eqref{eq:dual_genX_MIP} and the dual~\eqref{eq:dualGeneric-product} have the same feasible solutions, and therefore weak duality (Proposition~\ref{prop:weakDual}), strong duality (Propositions~\ref{prop:strongDual}) and complementary slackness (Corollary~\ref{cor:complSlack}) also hold. For this purpose, we need a result from the literature and a lemma.

\begin{thm}[\cite{jeroslow1978cutting, johnson1974group, Nemhauser}]\label{g_bar_ineq}
If $g:\rr^m\mapsto \rr$ is a subadditive function such that $g(0)=0$, then for all $ \omega\in\rr^m$ with $\overline{g}(\omega)<\infty$ and for all $\lambda\geq0$, we have that $g(\lambda \omega)\leq\lambda\overline{g}(\omega)$.
\end{thm}
\begin{lem} \label{lem:f iff f-bar}
Let $K \subseteq \rr^n$ be a cone, $g:\rr^m\to \rr$ be a subadditive function such that $g(0)=0$,  $N\in\rr^{m \times n}$ and $e\in\rr^n$. Then, $g(N\upsilon) \le e^T \upsilon$ for all $\upsilon\in K$ if and only if $\bar g(N\upsilon) \le e^T \upsilon$ for all $\upsilon\in K$.
\end{lem}
\begin{proof}
($\Rightarrow$): Let $\upsilon\in K$. Due to the definition of $\bar g$, we have
\(
\bar g(N\upsilon) 
= \lim\sup_{\delta\to0^+}\frac{ g (\delta N \upsilon)}{\delta} 
\le  \lim\sup_{\delta\to0^+}\frac{ \delta e^T \upsilon}{\delta} =  e^T \upsilon,
\)
where the inequality follows from $g ( N (\delta\upsilon)) \le  e^T (\delta\upsilon)$ since $\delta \upsilon \in K$.

($\Leftarrow$): Let $\upsilon\in K$. Then, due to Theorem~\ref{g_bar_ineq} with $\omega=N\upsilon$ and $\lambda=1$, we obtain $g(N\upsilon)  \le \bar g(N\upsilon)  \le e^T \upsilon$.
\end{proof}

\begin{prop}
\label{prop:Dual equivalence}
Let $\genX = (K_1 \cap \zz^{n_1}) \times K_2$. Suppose that $\cI( K_1 \cap \zz^{n_1} )$ and $ \cI(K_2)$ are generating sets of the sets $ K_1 \cap \zz^{n_1}$ and $K_2$, respectively. Then,
\(
 ( \cI( K_1 \cap \zz^{n_1} ) \times \{0\}^{n_2} ) \cup (\{0\}^{n_1} \times \cI(K_2) )
\) is a generating set of $\genX$. Moreover, dual problems~\eqref{eq:dual_genX_MIP} and \eqref{eq:dualGeneric-product} are equivalent.
\end{prop}
\begin{proof}
    Let $(x,y) \in \genX$. 
 Since $x\in K_1\cap \zz^{n_1}$ and $ y \in K_2$,  we have that there exist $u_1,\ldots,u_p\in \cI(K_1\cap \zz^{n_1} )$ and  $\lambda_1,\ldots,\lambda_p  \in \zz_+$ such that $x=\sum_{i=1}^p \lambda_i u_i$, and $v_1,\ldots,v_q\in \cI (K_2)$ and  $\mu_1,\ldots,\mu_q\in \zz_+$ such that $y=\sum_{j=1}^q \mu_j v_j$. 
 Then, we obtain 
 $(x,y) =  \sum_{i=1}^p \lambda_i (u_i,0) + \sum_{j=1}^q \mu_j (0,v_j)$, which proves the first assertion.

To prove the second assertion, it suffices to show that for $f\in \mathcal{F}_\cD$ with $f(0) = 0$  we have that 
\begin{equation} \label{dual1}
    \begin{aligned}
      &\hspace{0.5em} f(A u+Gv  ) \le c^T u+d^Tv & \textup{for all}\ (u,v)& \in \cI( (K_1 \cap \zz^{n_1}) \times K_2 ),\\
    \end{aligned}
    \end{equation}
 is equivalent to
\begin{equation} \label{dual2}
    \begin{aligned}
      &\hspace{0.5em} f(A u  ) \le c^T u & \textup{for all}\  u& \in \cI(K_1 \cap \zz^{n_1}  )\\
    &\hspace{0.5em} \bar f( G v) \le  d^Tv& \textup{for all}\  v& \in \cI( K_2 ),
    \end{aligned}
    \end{equation}

We will use the fact that by Lemma~\ref{lem:IntGenSet-indifferent}, the constraints of the dual~\eqref{eq:dual_genX_MIP} do not depend on the integral generating set used. Since a generating set of $\genX=(
K_1 \cap \zz^{n_1}) \times K_2$ is \(
 ( \cI( K_1 \cap \zz^{n_1} ) \times \{0\}^{n_2} ) \cup (\{0\}^{n_1} \times \cI(K_2) )
\), then if $f$ satisfies constraint~\eqref{dual1},  we obtain that $f(A u  ) \le c^T u$ for $(u,0)\in\cI( K_1 \cap \zz^{n_1} )\times \{0\}^{n_2}$ and $f( G v) \le  d^Tv$ for $(0,v)\in\{0\}^{n_1} \times \cI(K_2)$. Hence, by Lemma~\ref{lem:f iff f-bar}, we obtain that constraints~\eqref{dual2} are satisfied by $f$.
For the reverse direction, we assume that  $f$ satisfies~\eqref{dual2}. We will show that constraints~\eqref{dual1} are satisfied for $(u,v)\in 
 ( \cI( K_1 \cap \zz^{n_1} ) \times \{0\}^{n_2} ) \cup (\{0\}^{n_1} \times \cI(K_2))$. By the second constraint in~\eqref{dual2} and Lemma~\ref{lem:f iff f-bar}, we obtain that $\bar f( G v) \le  d^Tv$ implies that $f( G v) \le  d^Tv$. Therefore, by using subadditivity of $f$, we conclude that $f(Au+Gv)\leq f(Au)+f(Gv)\leq  c^T u+d^Tv$, as desired.
\end{proof}

\subsection{Conic MIPs with block structure}

In this section, we specialize our results 
for a conic MIP with block structure given as 
\begin{equation}\label{eq:genX_MIP-block}
\begin{aligned}
z^* :=   \inf  &\hspace{0.5em}   \sum_{\ell=1}^L \left [  {c^\ell}^T x^\ell + {d^\ell}^T y^\ell \right] \\
  \mathrm{s.t.}   &\hspace{0.5em}\sum_{\ell=1}^L  \left [ A^\ell x^\ell + G^\ell y^\ell  \right] = b \\
  & \hspace{0.5em}  (x^\ell,y^\ell)  \in  K^\ell \cap (\zz^{n_1^\ell} \times \rr^{n_2^\ell} ) \qquad \ell=1,\dots,L,
  \end{aligned}
\end{equation}
where $K^\ell \subseteq \rr^{n_1^\ell+n_2^\ell}$ is a regular cone and  {$c^\ell \in \RR^{n_1^\ell}$, $d^\ell \in \RR^{n_2^\ell}$}, $A^\ell \in \RR^{m \times n_1^\ell}$, $G^\ell \in \RR^{m \times n_2^\ell}$, $b \in \RR^{m}$,  for $\ell=1,\dots,L$.
Clearly,  conic MIP~\eqref{eq:genX_MIP-block} is {an instance of}   MIP~\eqref{eq:genX_MIP} with 
{$\genX = \bigtimes_{\ell=1}^L \big [ K^\ell \cap (\zz^{n_1^\ell} \times \rr^{n_2^\ell} ) \big ]$,} hence, Propositions~\ref{prop:weakDual} and~\ref{prop:strongDual}, and Corollary~\ref{cor:complSlack}  hold true. However, in this case, we can also generalize Proposition~\ref{prop:generatorLinear-withE} as follows:
\begin{prop}\label{prop:generatorLinear-withE-block} 
        Consider a feasible conic MIP  of the form \eqref{eq:genX_MIP-block}.  Let $a^\ell_i$ and $g^\ell_i$ respectively denote the $i$-th row of matrices $A^\ell$ and $G^\ell$.
        Define 
        $$E' := \left\{ \ell : 
        \begin{bmatrix}
            c^\ell \\ d^\ell
        \end{bmatrix}
        - 
        \begin{bmatrix}
            A^\ell \\ G^\ell
        \end{bmatrix}^T
        \alpha   \in {K^\ell_*} 
        \text{ and } 
        \begin{bmatrix}
            a_i^\ell \\ g_i^\ell
        \end{bmatrix}
        \in K_*^\ell \ \  \text{ for all } i=1,\dots,m \right\}
        \text{ and } E := \{1,\dots,L\} \setminus E',$$ for some $\alpha\in\rr^m$. Then, 
    \begin{equation} \label{eq:generatorFuncLinear-block}
    \begin{split}
    F_\alpha(\omega) = \alpha^T \omega - \max \bigg \{ &
    \sum_{l \in E} 
    \left [ ( {A^\ell}^T \alpha - c^\ell )^T x^\ell + ( {G^\ell}^T \alpha  - d^\ell)^T y^\ell  \right ]
    : \\
    & \sum_{\ell \in E}  \left [ A^\ell x^\ell + G^\ell y^\ell  \right] \le \omega,
    \quad (x^\ell,y^\ell)  \in  K^\ell \cap (\zz^{n_1^\ell} \times \rr^{n_2^\ell} ), \ \ell=1,\dots,L  \bigg \}.
    \end{split}
    \end{equation}
\end{prop}
\begin{proof}
For contradiction, assume that all  optimal solutions $(\hat x^{\ell}, \hat y^{\ell})$, $\ell=1,\dots,L$ to problem~\eqref{eq:generatorFuncLinear-block} have a nonzero $(\hat x^{\ell'}, \hat y^{\ell'})$ for some ${\ell'} \in E'$. 
Now, given such a solution, let us construct a new solution  $(\check x^{\ell}, \check y^{\ell})$ such that $(\check x^{\ell}, \check y^{\ell})= (\hat x^{\ell}, \hat y^{\ell})$ for  $\ell \neq \ell'$ and  $(\check x^{\ell'}, \check y^{\ell'})=(0,0)$.
Since  the matrix coefficients of each row of $\begin{bmatrix}
    A^{\ell'} & G^{\ell'}
\end{bmatrix}$ comes from the dual cone, the new solution $(\check x^{\ell'}, \check y^{\ell'})$ is also feasible. In addition, 
since the objective coefficient vector $  
        \begin{bmatrix}
            A^{\ell'} \\ G^{\ell'}
        \end{bmatrix}^T
        \alpha  - \begin{bmatrix}
            c^{\ell'} \\ d^{\ell'}
        \end{bmatrix} $ of the block variables $(x^{\ell'},y^{\ell'})$ comes from the negative of the dual cone, the objective function value of solution 
        $(\check x^{\ell}, \check y^{\ell})$
        is greater than or equal to that of solution 
        $(\hat x^{\ell}, \hat y^{\ell})$. This leads to the desired contradiction.
\end{proof}

\subsection{Conic integer programs defined by cones with $(R,G)$-finitely generated integral vectors}
\label{sec:pure_integer_(R,G)}
Unlike the case of rational polyhedral cones (see \cite{schrijver1998theory}), integral vectors in general convex cones may not have a finite integral generating sets. In order to overcome this, \cite{deloera2024integerpointsarbitraryconvex} introduced the concept of the monoid of integer points inside a convex cone to be {\em $(R,G)$-finitely generated} and studied it in the particular case of the second order cone and the positive semidefinite cone; some applications of their definition to the theory of conic integer programming are also discussed. We give their definition below.

\begin{dfn}[Group Action, $(R,G)$-finitely generated]
Given a group $G$ and a regular cone $K$, a group action is a function from $G\times K$ to $K$, that is, $(g,k)\mapsto g\cdot k$ satisfying: (1) Identity:  $I\cdot k = k$ for all $k\in K$, and (2) Compatibility: $g_2\cdot(g_1\cdot k)=(g_2g_1)\cdot k$ for all $g_1,g_2\in G$ and $k\in K$.

We say that $K \cap \zz^n$ is $(R,G)$-finitely generated if there is a finite subset $R \subseteq K\cap \zz^n$ and a finitely generated subgroup $G \subseteq \{U\in\zz^{n\times n}\tq|\det(U)|=1\}$ acting on $K$ linearly such that
    \begin{enumerate}
        \item both the cone $K$ and $K\cap \zz^n$ are invariant under the group action, i.e., $G\cdot K =K$ and $G\cdot (K\cap \zz^n) =K\cap \zz^n$, and
        \item for any $z\in K\cap \zz^m$ there exists $r_1,\ldots,r_p\in R$, $g_1,\ldots,g_p\in G$ and $\lambda_1,\ldots,\lambda_p\in\zz_+^m$ such that 
  $z=\sum_{i=1}^p\lambda_ig_i\cdot r_i$.
    \end{enumerate}
\end{dfn} 
It follows from the definition above that the set $\cI(G,R)=\{g\cdot r\tq g\in G,r\in R\}$ is an integral generating set of $\genX=K\cap \zz^n$. Therefore, in the pure integer case ($n_1=n, n_2=0$), the subadditive dual problem~\eqref{eq:dual_genX_MIP} takes the form:
\begin{equation}\label{eq:dualGeneric_(R,G)fg}
    \begin{aligned}
    \rho^* := \sup  &\hspace{0.5em}  f(b) \\
      \mathrm{s.t.}   &\hspace{0.5em} f(A (g\cdot r)) \le c^T(g\cdot r)  & g\in G,r\in R\\
      & \hspace{0.5em}  f(0) = 0, \ f \in \mathcal{F}_\cD.
    \end{aligned}
    \end{equation}
The fact that the integral generating set in the dual~\eqref{eq:dualGeneric_(R,G)fg} is $(R,G)$-finitely generated may lead to some algorithmic ideas to solve this dual or to use its feasible dual functions to generate cuts for the associated primal problem.



\section{Illustrations of our results}
\label{sec:illustration}
\subsection{A numerical example}
In this section, we give a numerical example and illustrate the applicability of our results.
For this purpose, we consider the conic MIP
\begin{equation}\label{eq:illustrationConicMIP}
 \min_{(x,y)\in\zz\times\rr^2} \left\{ x + y_1 + y_2 : \ x\ge \|y\|_2, \ 4x+y_1=5 \right\},
\end{equation}
whose feasible region is the singleton $(1,1,0)$.
It can be checked that for the problem~\eqref{eq:illustrationConicMIP}, we have that $\Omega:= \{w\tq 4x+y_1 = \omega\ \textup{for}\ (x,y)\in \boldsymbol{L}^{3} \cap (\zz\times\rr^2)\}=\{0\} \cup [3,5] \cup [6,\infty)$ and $\Omega^\leq 
= \{w\tq 4x+y_1 \le \omega\ \textup{for}\ (x,y)\in \boldsymbol{L}^{3} \cap (\zz\times\rr^2)\}=\rr_+$.

Notice that  the conic dual of the continuous relaxation of the conic MIP~\eqref{eq:illustrationConicMIP}, which is given  below, is feasible (e.g. $\alpha=-1$ is a feasible solution):
\[
\max_{ \alpha\in\rr } \left\{ 5 \alpha  :
\begin{bmatrix}
    4\alpha \\ \alpha \\ 0
\end{bmatrix}
\preceq_{\boldsymbol{L}^{3}}
\begin{bmatrix}
    1 \\ 1 \\ 1
\end{bmatrix}
\right\}.
\]
Here, we define the Lorentz cone in dimension $n+1$ as $\boldsymbol{L}^{n+1} := \{(x,y)\in\rr\times\rr^n : x \ge \|y\|_2 \}$. 
Therefore, due to Proposition~\ref{prop:strongDual}, the subadditive dual given below is a strong dual (notice that in this case we can choose 
$  \cI( {\boldsymbol{L}^{3}} \cap (\zz^{1}\times\rr^{2}) ) = \{ 1 \} \times D $, where $D:=\{v\in\rr^2 : \| v \|_2 \le 1\}$): 
\begin{equation}\label{eq:illustrationConicMIP-dual}
    \max  \left \{ f(5) : \  f(4 + v_1 ) \le 1 + v_1 + v_2 \ \forall v \in D,    \ f(0) = 0, \ f \in \mathcal{F}_{\rr^3}. \right \}
    \end{equation}

Moreover, as we will show below, there exists an optimal generator  function to the subadditive dual~\eqref{eq:illustrationConicMIP-dual}. To this end, we first point out that the set
\(
\conv(\{ (x,y)\in {\boldsymbol{L}^{3}} \cap (\zz\times\rr^2) : \  4x+y_1 \le 5 \} ) = \{ (x,y)\in {\boldsymbol{L}^{3}} : x \le 1  \}
\) 
 is conic representable.
Since the later set is also bounded, the conic dual of the program $\min\big\{ x + y_1 + y_2 :  \ 4x+y_1=5 , \ x\le 1, \ (x,y)\in {\boldsymbol{L}^{3}}    \big\} $, which is given below, is a strong dual:
\[
\max_{ (\alpha,\beta)\in\rr\times\rr_+ } \left\{ 5 \alpha - \beta :
\begin{bmatrix}
    4\alpha-\beta \\ \alpha \\ 0
\end{bmatrix}
\preceq_K
\begin{bmatrix}
    1 \\ 1 \\ 1
\end{bmatrix}
\right\}.
\]
In fact, it can be shown that in an optimal solution of the dual problem, we have   $\alpha^*=10445$.
Now, we can apply Theorem~\ref{thm:strong+generator} to obtain an optimal subadditive generator function as follows:
\begin{equation} \label{eq:genFunction}
F_{10445}(\omega) = 
10445 \omega - \sup_{ (x,y)\in {\boldsymbol{L}^{3}} \cap (\zz\times\rr^2) } \left \{ 41779 x   + 10444y_1  - y_2 :\,  4x+y_1 \le \omega  \right \}.\end{equation}
In this case, it is possible to compute the generator function over $\Omega^\leq$ based on three cases:
\begin{itemize}
    \item $\omega=0$: In this case, $x^*=y_1^*=y_2^*=0$ and $F_{10445}(0) = 0$.
    \item $\omega\in(0,3)$: In this case, $x^*=y_1^*=y_2^*=0$ and $F_{10445}(\omega) = 10445\omega$.
    \item $\omega\in[3,5]$: In this case,  $x^*=1, y_1^*=\omega-4, y_2^*=-\sqrt{1-(\omega-4)^2}$ and $F_{10445}(\omega) = -3 + \omega- \sqrt{1-(\omega-4)^2} $.
    \item $\omega\in(5,6)$: In this case,  $x^*=1, y_1^*=1, y_2^*=0$ and $F_{10445}(\omega) = 2 + 10445(\omega-5)$.
    \item $\omega \ge 6$: In this case, we first obtain the solution of the continuous relaxation as $x'=\frac{\sqrt{0.06}+1.6}{6}w$ after eliminating $y_1$ and $y_2$ variables. For the conic MIP, an optimal value of $x$ is   either $x^*=\lceil x' \rceil$ or $x^*=\lfloor x' \rfloor$. 
    After plugging in the values of $y_1^*=\omega-4x^*$ and $y_2^*= -\sqrt{(x^*)^2-(y_1^*-4)^2} $, we obtain
    $$F_{10445}(\omega) = \max\big\{ -3\lceil x' \rceil + \omega - \sqrt{ \lceil x' \rceil^2 -(\omega-4\lceil x' \rceil)^2} ,  -3\lfloor x' \rfloor + \omega - \sqrt{ \lfloor x' \rfloor^2 -(\omega-4\lfloor x' \rfloor)^2}
    \big\}.$$    
    
\end{itemize}
The value of the generator function $F_{10445}$ as computed above  and the value of the
function $F'_{10445}$, which is  obtained by relaxing the integrality restriction of $x$ variable in equation~\eqref{eq:genFunction}, are drawn in Figure~\ref{fig:genFunc}  for different r.h.s. values of $\omega\in\Omega$. We also note that  $F'_{10445}$ provides an under-approximation of $F_{10445}$, as expected.  


\input{figure}

\subsection{Examples of sets $\genX^\leq(b)$ with conic representable integer hull}

Recall the set
$\genX^\leq(b)=\{(x,y) \in\genX \tq Ax+Gy\leq b\}$ with $\genX=K   \cap   (\mathbb{Z}^{n_1} \times \mathbb{R}^{n_2})$. When $K$ is the cartesian product of linear transformations of nonnegative orthants and the data is rational, by Meyer's theorem~\cite{Meyer1974},  the integer hull of $\genX^\leq(b)$ is a rational polyhedron (see~\cite{dey2013some} for more conditions for this to happen). 
We now present examples of sets $\genX^\leq(b)$, where the cone $K$ is a not necessarily a rational polyhedral cone and such that their integer hull is conic representable, and so the main assumption in Theorem~\ref{thm:strong+generator} is satisfied.

 

\subsubsection{Conic sets with finitely many integer fibers}
\label{sec:finitely_many_fibers}


Let $b\in \Omega^\leq$ and $ \mathcal{U} \subseteq\zz^{n_1}$ be the projection onto $\rr^{n_1}$ of $\genX^\leq(b)$. For $u \in \mathcal{U}$, the integer fiber defined by $u$ is the convex set 
$$\genXcont^\leq(b,u)=\{(x,y) \in K \cap (\mathbb{Z}^{n_1} \times \mathbb{R}^{n_2})\tq Gy\leq b-Au,\ x=u\}.$$ 
Assume that $ \mathcal{U} $ is a finite set (note that the condition of $ \mathcal{U} $ being a finite set is always satisfied for conic mixed-integer binary programs) and that all its integer fibers are conic representable, that is, for each $u\in\mathcal{U}$, there exists a regular cone $C^u$, matrices $\Pi^u,\Phi^u,\Psi^u$ and a vector $\pi^u$ such that
$$\genXcont^\leq(b,u)=\{(x,y)\in K   \cap   (\mathbb{Z}^{n_1} \times \mathbb{R}^{n_2})\tq\exists\, w^u\in \rr^{n_3^u}\ \textup{s.t.}\ \Pi^u x  + \Phi^u y + \Psi^u w^u \succeq_{C^u} \pi^u,\ x=u\}.$$
Observe that 
\begin{equation}\label{eq:conv_equals_union}
\conv(\genX^\leq(b))=\conv\left(\bigcup_{u\in \mathcal{U}}\genXcont^\leq(b,u)\right).    
\end{equation}

Furthermore, the recession cone of each set $\genXcont^\leq(b,u)$ is equal to $\genXcont^\leq(0,0)=\{(x,y) \in K \cap (\mathbb{Z}^{n_1} \times \mathbb{R}^{n_2})\tq Gy\leq 0,\ x=0\}$, so we conclude that the convex hull of the union in~\eqref{eq:conv_equals_union} is a closed set.
The closure of the convex hull of the (finite) union of conic representable sets is known to be conic representable, and from~\eqref{eq:conv_equals_union},  we obtain the following explicit formulation for $\conv(\genX^\leq(b))$ (see~\cite{BenTal01}):

\[
\begin{split}
\conv(\genX^\leq(b))=\bigg\{(x,y)\in K   \cap   (\mathbb{Z}^{n_1} \times \mathbb{R}^{n_2})\,  :\,\exists\, & (x^u, y^u, w^u , \lambda^u)\in (\mathbb{Z}^{n_1} \times \mathbb{R}^{n_2} \times \rr^{n_3^u} \times \rr_+)\, :\, \\ 
& \Pi^u x^u  + \Phi^u y^u + \Psi^u w^u \succeq_{C^u} \pi^u \lambda^u,\ x^u = u \lambda^u,\ \textup{for all}\  u\in \mathcal{U} \\
& x = \sum_{u\in \mathcal{U} } x^u , \ y = \sum_{u\in \mathcal{U} } y^u, \ 1 = \sum_{u\in \mathcal{U} } \lambda^u
\bigg\}.
\end{split}
\]

\subsubsection{Packing conic mixed-integer programs}


A definition of  packing conic sets   is given in~\cite{ajayi2021strongdualitytheoremsalternative} and a definition of packing semidefinite programs involving nonlinear mappings is given in~\cite{iyengar2011PackingSDP}. In our setting, the definition given in~\cite{ajayi2021strongdualitytheoremsalternative} is as follows.

\begin{dfn}[Packing conic set]
    A packing conic set is a set $\genXcont^\leq(b)=\{(x,y) \in K    \tq Ax+Gy\leq b\}$ such that rows of $[A\ G]$ belong to the dual cone $K_*$.
\end{dfn}

It can be shown that for packing conic sets, we have that $\Omega^\leq=\rr_+^m$, that is, the set 
$\genX^\leq(\omega)$ is feasible for any $\omega\geq0$.
If the set $\conv(\genX^\leq(b))$ is bounded, then the generating subadditive functions $F_\alpha$ are well-defined. In addition, by the discussion in Section~\ref{sec:finitely_many_fibers}, the set $\conv(\genX^\leq(b))$ is conic representable, so we can apply the results in Theorem~\ref{thm:strong+generator} to conclude that there exists $\alpha^*\in\rr^m$ such that the function $F_{\alpha^*}:\rr_+^m \to \rr$ is an optimal solution to the corresponding subadditive dual~\eqref{eq:dual_genX_MIP}. 
We now give a characterization for a packing conic set to be a bounded set. For some sufficient conditions for boundedness, see~\cite{ajayi2021strongdualitytheoremsalternative}.

    %
    \begin{prop}\label{prop:boundedPacking}
    Consider a packing conic set of the form $\genXcont^\leq(b)$. Then, there exists a vector  $\fancymu \ge 0$ such that $\begin{bmatrix}
        A & G
    \end{bmatrix}^T \fancymu  \in \textup{int}(K_*)$ if and only if the set $\genXcont^\leq(b)$ is a bounded set.
    \end{prop}
    \begin{proof} Note that the recession cone of $\genXcont^\leq(b)$ is the set the $\genXcont^\leq(0)=\{(x,y) \in K   \tq Ax+Gy\leq 0\}$ and that $\genXcont^\leq(b)$ is bounded if and only if $\genXcont^\leq(0)=\{0\}$. For any $(k_x, k_y)\in \textup{int}(K_*)$, we can consider the primal problem
    $\sup\{k_x^Tx+k_y^Ty\tq (x,y)\in \genXcont^\leq(0)\}$ and its dual $\inf\left\{0\tq \begin{bmatrix}A&G\end{bmatrix}^T\mu\succeq_{K_*}(k_x, k_y),\ \mu\geq0\right\}.$

    ($\Rightarrow$): Assume that  there exists a vector  $\fancymu \ge 0$ such that $ (k_x, k_y) :=\begin{bmatrix}
        A & G
    \end{bmatrix}^T \fancymu  \in \textup{int}(K_*)$. This implies that the dual problem is feasible with the optimal value of 0. By weak duality, we conclude that  $k_x^Tx+k_y^Ty \le 0$ for any $(x,y)\in K$. Considering the fact that $(k_x, k_y) \in \int (K_*)$, we conclude that $k_x^Tx+k_y^Ty = 0$. This implies that $(x,y)=0$ is the only primal feasible solution, implying that $\genXcont^\leq(b)$ is a bounded set.

    ($\Leftarrow$): Assume that $\genXcont^\leq(b)$ is a bounded set, then $\genXcont^\leq(0)=\{0\}$. Then, for any $(k_x, k_y)\in \textup{int}(K_*)$ the primal problem must have an optimal value of $0$. By strong duality, we must have that the dual is feasible, and therefore, there exists a vector  $\fancymu \ge 0$ such that $\begin{bmatrix}
        A & G
    \end{bmatrix}^T \fancymu - (k_x, k_y)\in K_*$, which implies that $\begin{bmatrix}
        A & G
    \end{bmatrix}^T \fancymu  \in \textup{int}(K_*)$.
    \end{proof}

\subsection{An example of a conic MIP with block structure: clustering problem}

We now give an example for conic MIP with block structure of the form~\eqref{eq:genX_MIP-block}. 
Suppose that we are given points $\xi_\iota \in \mathbb{R}^p$, $\iota=1,\dots,I$. Our aim is to group these points into $Q$ clusters and find one representative $\chi^q  \in \mathbb{R}^p$ for each cluster, $q=1,\dots,Q$ in such a way that the total   Euclidean distance of the points assigned to a cluster to their representative is minimized. This problem can be formulated  as the following optimization problem
\[
 \min_{ \zeta_{\iota}^q \in \{0,1\}, \chi^q \in \mathbb{R}^{p} } \left\{ \sum_{q=1}^Q \sum_{\iota=1}^I  \zeta_{\iota}^q \| \chi^q - \xi_\iota \|_2 : \sum_{q=1}^Q \zeta_{\iota}^q = 1 , \ \iota=1,\dots,I \right\},
\]
or equivalently as
\begin{subequations} \label{eq:reformulation}
\begin{align*}
 \min_{ \zeta, \chi, \delta } \left \{  \sum_{q=1}^Q \sum_{\iota=1}^I\delta_{\iota}^q 
 : 
\sum_{q=1}^Q \zeta_{\iota}^q = 1 , 
\  \| \chi^q - \xi_\iota \|_2 \le \delta_{\iota}^q + M(1-\zeta_{\iota}^q) ,  
\  \zeta_{\iota}^q \in \{0,1\}, \ \iota =1,\dots,I; q=1,\dots,Q 
\right \},
\end{align*}
\end{subequations}
with $M = \max_{\iota,\iota'} \|\xi_\iota - \xi_{\iota'}\|$. To put this formulation in the form of~\eqref{eq:genX_MIP-block}, we define new variables as in the below formulation: 
\begin{subequations} \label{eq:rereformulation}
\begin{align*}
 \min_{\zeta, \chi, \delta, \bar\zeta, \bar\eta, \eta}  \ & \sum_{q=1}^Q  \sum_{\iota=1}^I\delta_{\iota}^q 
 \\
\text{s.t.} \ & \sum_{q=1}^Q \zeta_{\iota}^q = 1 
, \ 
\eta_{\iota}^q = \chi^q - \xi_\iota, \ \bar\eta_{\iota}^q = \delta_{\iota}^q + M(1-\zeta_{\iota}^q), \
\zeta_{\iota}^q + \bar\zeta_{\iota}^q = 1 \quad &\iota&=1,\dots,I; q=1,\dots,Q 
\\
\ & \bar\eta_{\iota }^q \ge \| \eta_{\iota}^q \|_2 , \ \zeta_{\iota}^q \ge 0, \ \bar\zeta_{\iota }^q \ge 0    \quad &\iota&=1,\dots,I; q=1,\dots,Q 
\\
\ & \zeta_{\iota}^q \in \zz, \bar\zeta_{\iota}^q \in \zz, \chi^q \in \mathbb{R}^p, \delta_{\iota}^q \in \rr, \bar \eta_\iota^{q} \in \mathbb{R} , \eta_\iota^{q} \in \mathbb{R}^p
 \quad &\iota&=1,\dots,I; q=1,\dots,Q 
 ,
\end{align*}
\end{subequations}
Now, this formulation is in the form~\eqref{eq:genX_MIP-block} with $\ell=q$, 
$x^q = [(\zeta_\iota^q)_\iota, (\bar\zeta_\iota^q)_\iota] \in \zz_+^I \times \zz_+^I $ and 
$y^q = [(\chi^q, (\delta_\iota^q)_\iota,  (\bar \eta_\iota^q, \eta_\iota^q)_\iota] \in \rr^p   \times \rr^I\times {\boldsymbol{L}^{p+1,I}}    $, where ${\boldsymbol{L}^{p+1,I}}$ is the product of $I$ many Lorentz cones in $\rr^{p+1}$.

\section*{Statements and Declarations}

\subsection*{Competing interests}
The authors declare that they have no competing interests.

\subsection*{Funding}
 We would  like to thank for the support from the  ANID grant  Fondecyt \# 1210348 and TUBITAK visiting researcher grant under Program 2221.










\bibliographystyle{plain}
\bibliography{references}

\begin{thebibliography}{10}

\bibitem{ajayi2021strongdualitytheoremsalternative}
T.~Ajayi, A.~Gupte, A.~Khademi, and A.~Schaefer.
\newblock On strong duality, theorems of the alternative, and projections in
  conic optimization, 2021.

\bibitem{AJAYI2020329}
T.~Ajayi, V.~Suriyanarayana, and A.~J. Schaefer.
\newblock Theorems of the alternative for conic integer programming.
\newblock {\em Operations Research Letters}, 48(3):329--335, 2020.

\bibitem{aspman2024taming}
J.~Aspman, G.~Korpas, and J.~Marecek.
\newblock Taming binarized neural networks and mixed-integer programs.
\newblock In {\em Proceedings of the AAAI Conference on Artificial
  Intelligence}, volume~38, pages 10935--10943, 2024.

\bibitem{BenTal01}
A.~Ben-Tal and A.~Nemirovski.
\newblock {\em Lectures on Modern Convex Optimization}.
\newblock Society for Industrial and Applied Mathematics, 2001.

\bibitem{cheung2016certificates}
K.~K.~H. Cheung and B.~Moazzez.
\newblock Certificates of optimality for mixed integer linear programming using
  generalized subadditive generator functions.
\newblock {\em Advances in Operations Research}, 2016, 2016.

\bibitem{DMM2024}
S.~S. Dey, F.~Meunier, and D.~A. Mor{\'{a}}n~R.
\newblock Geoffrion's theorem beyond finiteness and rationality.
\newblock 2024.

\bibitem{DM2013}
S.~S. Dey and D.~A. Mor{\'{a}}n~R.
\newblock Some properties of convex hulls of integer points contained in
  general convex sets.
\newblock {\em Mathematical Programming}, 141(1):507--526, Oct 2013.

\bibitem{dey2013some}
S.~S. Dey and D.~A. Mor{\'a}n~R.
\newblock Some properties of convex hulls of integer points contained in
  general convex sets.
\newblock {\em Mathematical Programming}, 141(1):507--526, 2013.

\bibitem{Flippo1996}
O.~E. Flippo and A.~H.~G. Rinnooy~Kan.
\newblock Additively separable duality theory.
\newblock {\em Journal of Optimization Theory and Applications},
  88(2):381--397, Feb 1996.

\bibitem{geoffrion2010lagrangian}
A.~M. Geoffrion.
\newblock Lagrangian relaxation for integer programming.
\newblock In {\em 50 Years of Integer Programming 1958-2008}, pages 243--281.
  Springer, 2010.

\bibitem{Guzelsoy}
M.~Guzelsoy and T.~K. Ralphs.
\newblock Duality for mixed-integer linear programs.
\newblock {\em The International Journal of Operations Research}, pages
  118--137, 2007.

\bibitem{hemmecke2007representation}
R.~Hemmecke and R.~Weismantel.
\newblock Representation of sets of lattice points.
\newblock {\em SIAM Journal on Optimization}, 18(1):133--137, 2007.

\bibitem{iyengar2011PackingSDP}
G.~Iyengar, D.~J. Phillips, and C.~Stein.
\newblock Approximating semidefinite packing programs.
\newblock {\em SIAM Journal on Optimization}, 21(1):231--268, 2011.

\bibitem{jeroslow1978cutting}
R.~Jeroslow.
\newblock Cutting-plane theory: Algebraic methods.
\newblock {\em Discrete Mathematics}, 23(2):121--150, 1978.

\bibitem{jeroslowmonoids1978}
R.~G. Jeroslow.
\newblock Some basis theorems for integral monoids.
\newblock {\em Mathematics of Operations Research}, 3(2):145--154, 1978.

\bibitem{jeroslow1979minimal}
R.~G. Jeroslow.
\newblock Minimal inequalities.
\newblock {\em Mathematical Programming}, 17(1):1--15, 1979.

\bibitem{johnson1974group}
E.~L. Johnson.
\newblock On the group problem for mixed integer programming.
\newblock {\em Mathematical Programming Study}, (2):137--179, 1974.

\bibitem{klabjan2004practical}
D.~Klabjan.
\newblock A practical algorithm for computing a subadditive dual function for
  set partitioning.
\newblock {\em Computational optimization and applications}, 29:347--368, 2004.

\bibitem{Klabjan2007}
D.~Klabjan.
\newblock Subadditive approaches in integer programming.
\newblock {\em European Journal of Operational Research}, 183(2):525 -- 545,
  2007.

\bibitem{kocuk2019subadditive}
B.~Kocuk and D.~A. Mor{\'{a}}n~R.
\newblock On subadditive duality for conic mixed-integer programs.
\newblock {\em SIAM Journal on Optimization}, 29(3):2320--2336, 2019.

\bibitem{deloera2024integerpointsarbitraryconvex}
J.~A.~De Loera, B.~Marsters, L.~Xu, and S.~Zhang.
\newblock Integer points in arbitrary convex cones: The case of the {PSD} and
  {SOC} cones, 2024.

\bibitem{Meyer1974}
R.~R. Meyer.
\newblock On the existence of optimal solutions to integer and mixed-integer
  programming problems.
\newblock {\em Mathematical Programming}, 7(1):223--235, Dec 1974.

\bibitem{Moran}
D.~A. Mor{\'{a}}n~R., S.~S. Dey, and J.~P. Vielma.
\newblock A strong dual for conic mixed-integer programs.
\newblock {\em SIAM Journal on Optimization}, 22(3):1136--1150, 2012.

\bibitem{Nemhauser}
G.~L. Nemhauser and L.~A. Wolsey.
\newblock {\em Integer and Combinatorial Optimization}.
\newblock Wiley-Interscience, New York, NY, USA, 1988.

\bibitem{santana2017some}
A.~Santana and S.~S. Dey.
\newblock Some cut-generating functions for second-order conic sets.
\newblock {\em Discrete Optimization}, 24:51--65, 2017.

\bibitem{schrijver1998theory}
A.~Schrijver.
\newblock {\em Theory of linear and integer programming}.
\newblock John Wiley \& Sons, 1998.

\bibitem{Tind1981}
J.~Tind and L.~A. Wolsey.
\newblock An elementary survey of general duality theory in mathematical
  programming.
\newblock {\em Mathematical Programming}, 21(1):241--261, Dec 1981.

\end{thebibliography}

\end{document}